\documentclass[11pt]{Paper}

 \usepackage[foot]{amsaddr}
\usepackage[framemethod=TikZ]{mdframed}
\usepackage{lipsum}

\usepackage{fancybox}
 \usepackage{enumerate}
  
\usepackage{cite}
\usepackage{scalerel}
\usepackage{amsthm,lipsum}
\usepackage{amsmath}
\usepackage{amssymb}
  \usepackage{paralist}
  \usepackage{graphics} 
  \usepackage{epsfig} 
  \usepackage{csquotes}
  \usepackage{subcaption}
\usepackage{graphicx}  
\usepackage{epstopdf}
 \usepackage[colorlinks=true]{hyperref}
\usepackage{bm}
\usepackage{siunitx}
\hypersetup{urlcolor=blue, citecolor=red}
\usepackage[margin=1.2in]{geometry}
\AfterEndEnvironment{figure}{\noindent\ignorespaces}

\newtheorem{theorem}{Theorem}[section]
\newtheorem{corollary}[theorem]{Corollary}

\newtheorem{proposition}[theorem]{Proposition}

\theoremstyle{definition}
\newtheorem{remark}{Remark}[section]

\usepackage[T1]{fontenc}    
\usepackage[utf8]{inputenc} 

\renewenvironment{proof}{{\bf \noindent Proof.}}{\hfill $\blacksquare$\vspace{0.25cm}}

\def\f{\longrightarrow}

\def\N{\mathbb{N}}

\def\e{\varepsilon}

\def\x{\bar{x}}
\def\y{\bar{y}}
\def\a{\bar{a}}
\def\<{\langle}
\def\>{\rangle}
\def\r{\bar{r}}

\def\L{\mathcal{L}}

\def\e{\varepsilon}
\def\en{\varepsilon_n}
\def\R{\mathbb{R}}

\def\P{\mathcal{P}}

\def\rhs{\textnormal{RHS}}
\def\O{\mathcal{O}}
\def\inte{\textnormal{int}\,}
\def\clo{\textnormal{cl}\,}
\def\epi{\textnormal{epi}\,}
\def\bdry{\textnormal{bdry}\,}

\def\proj{\textnormal{proj}\,}

\def\rhs{\textnormal{RHS}\sp}

\def\ax{{a_{x}}}
\def\nvarphi{\bp\nvarphi}
\def\sp{\hspace{0.015cm}}
\def\bp{\hspace{-0.08cm}}

\def\bbbp{\hspace{-0.01cm}}
\def\up{\hspace{-0.003cm}\scaleto{U\hspace{-0.068cm}P}{3.2pt}}

\def\rb{\bp r}
\def\sr{\hspace{-0.005cm}\scaleto{(r)}{4.8pt}}

\begin{document}

\title[The Variable Radius Form of the Extended Exterior Sphere Condition]{The Variable Radius Form of the Extended Exterior Sphere Condition}

\author[C. Nour and J. Takche]{}
\author{Chadi Nour$^*$}
\address{Department of Computer Science and Mathematics, Lebanese American University, Byblos Campus, P.O. Box 36, Byblos, Lebanon}
\email{cnour@lau.edu.lb}
\thanks{$^*$Corresponding author}

\author{Jean Takche}
\email{jtakchi@lau.edu.lb}




\dedicatory{This paper is dedicated to Professor R. T. Rockafellar on the occasion of his 90th birthday\vspace*{-0.25cm}}

\begin{abstract} We introduce a {\it variable} radius form of the {\it extended exterior sphere condition} of \cite{JCA2024}, and then, we prove that the complement of a closed set satisfying this new property is nothing but the union of closed balls with {\it lower semicontinous} radius function. This generalizes, to the variable radius case, the main result of \cite{JCA2024}, namely, \cite[Theorem 1.2]{JCA2024}. On the other hand, as it is shown in \cite{JCA2018b,JCA2022} for prox-regularity, the exterior sphere condition, and the union of closed balls property, we prove that the constant and the variable radius forms of the extended exterior sphere condition belong to the {\it $S$-convexity} regularity class.
\end{abstract}
\maketitle
\vspace{-0.7cm}
\section{Introduction} Let $A$ be a nonempty and closed subset of $\R^n$, and denote by $\bdry A$ and $\inte A$ the boundary and the interior of $A$, respectively. For $r\colon\bdry A\f (0,+\infty]$ a continuous function, we say that $A$ satisfies the {\it extended exterior $r(\cdot)$-sphere condition} if for every $a\in\bdry A$, the following assertions hold:
\begin{enumerate}[$\bullet$]
\item If $a\in\bdry(\inte A)$ then there {\it exists} a unit vector $\zeta_a\in N^P_A(a)$ such that $$\begin{cases}B(a+r(a)\zeta_a;r(a))\cap A=\emptyset,&\hbox{if}\; r(a)<+\infty,\\ B(a+\rho\zeta_a;\rho)\cap A=\emptyset\;\,\hbox{for all}\;\rho>0,&\hbox{if}\; r(a)=+\infty,\end{cases}$$
\item  If $s\not \in\bdry(\inte A)$ then for {\it all} unit vectors $\zeta_a\in N^P_A(a)$, we have $$\begin{cases}B(a+r(a)\zeta_a;r(a))\cap A=\emptyset,&\hbox{if}\; r(a)<+\infty,\\ B(a+\rho\zeta_a;\rho)\cap A=\emptyset\;\,\hbox{for all}\;\rho>0,&\hbox{if}\; r(a)=+\infty,\end{cases}$$
\end{enumerate}
where $N^P_A(a)$ denotes the {\it proximal normal cone} to $A$ at $a$, and $B(y;\sigma)$ is the open ball centered at $y$ with radius $\sigma>0$. 

When $r(\cdot)=r>0$ is a constant function, the above property coincides with the extended exterior $r$-sphere condition introduced and studied in \cite{JCA2024}. In fact, it was shown in  \cite[Theorem 1.2]{JCA2024} that the complement of a nonempty and closed set $A$ satisfying the extended exterior $r$-sphere condition is nothing but the union of closed balls with common radius $\frac{r}{2}$. This generalized, in finite-dimensional setting, \cite[Theorem 3]{Nacry} where the authors proved a similar result by assuming that $A$ is {\it $r$-prox-regular} which is {\it stronger} than the extended exterior $r$-sphere condition, see \cite[Section 2]{JCA2009}. The first main result of this paper is Theorem \ref{th1} in which we generalize \cite[Theorem 1.2]{JCA2024} to the variable radius case. More precisely, we prove in Theorem \ref{th1} that if  $A$ satisfies the extended exterior $r(\cdot)$-sphere condition, then $A^c$, the complement of $A$, is nothing but the union of closed balls with the radius function $\rho\colon A^c\f(0,+\infty]$ defined by \begin{equation} \label{radiusrho} \rho(x):=\min\left\{\frac{r(a)}{2} : a\in\proj_{A}(x)\right\},\end{equation} where $\proj_{A}(x)$ denotes the projection of $x$ on the boundary of $A$. By this union, we mean that for every $x\in A^c$, there exists $y_x\in A^c$ such that:
$$\begin{cases}x\in\bar{B}(y_x;\rho(x))\subset A^c,&\hbox{if}\; \rho(x)<+\infty,\\ x\in\bar{B}(x+\delta(y_x-x);\delta)\subset A^c\;\,\hbox{for all}\;\delta>0,&\hbox{if}\; \rho(x)=+\infty,\end{cases}$$ where $\bar{B}(y;\sigma)$ is the closed ball centered at $y$ with radius $\sigma$. One can easily see that our Theorem \ref{th1} is a generalization of \cite[Theorem 1.2]{JCA2024}. Indeed, if $r(\cdot)=r>0$ is a constant, then the function $\rho(\cdot)$ becomes the constant function $\frac{r}{2}$, and hence, $A^c$ is the union of closed balls with common radius $\frac{r}{2}$. Although going from fixed radius to continuous variable radius may be seen as a small generalization, it requires a {\it significant} revision of the technical approach used in \cite{JCA2024}. This can be readily inferred from the proof of Theorem \ref{th1} given in Section \ref{mainresult1}. Note that working with fixed radius on a neighborhood of a boundary point, then taking the limit as the size of the neighborhood goes to zero, is not a successful attempt to prove Theorem \ref{th1}. Indeed, one of the major obstacles caused by this technique is that it forces the replacement, in the final result, of $A^c$ by $\clo (A^c)$, the closure of the complement of $A$, which makes the reachable result {\it weaker} than Theorem \ref{th1}. 

On the other hand, when $A$ is {\it regular closed}, that is, $A$ equals to the closure of its interior, the extended exterior $r(\cdot)$-sphere condition coincides, for $\theta(\cdot):=\frac{1}{2r(\cdot)}$, with the {\it $\theta$-exterior sphere condition} introduced in \cite{NA2010}, and studied in depth in the papers \cite{NA2010,DC2011,JCA2011,JOTA2012,JCA2022}. For instance, it is proved in \cite{JCA2022} that the {\it $S$-convexity} property, introduced in \cite{JCA2018}, covers several known regularity properties including the $\varphi$-convexity, the $\theta$-exterior sphere condition, and the $\psi$-union of closed balls property, see \cite[Theorems 3.7,  3.15\,\&\,3.22]{JCA2022}.\footnote{For more information about $\varphi$-convexity and related properties such as {\it positive reach}, {\it proximal smoothness}, {\it prox-regularity}, and {\it $p$-convexity}, see \cite{canino,csw,cm,fed,prt,shapiro}} The extension of these latter results to encompass the extended exterior $r(\cdot)$-sphere condition will be done in  Theorem \ref{th2}, which is the second main result of this paper.

The layout of the paper is as follows. Our notations and some geometric definitions will be given in the next section. In Section \ref{mainresult1}, we prove that the complement of a nonempty and closed set satisfying the extended exterior $r(\cdot)$-sphere condition is nothing but the union of closed balls with the radius function $\rho(\cdot)$  defined in \eqref{radiusrho}. Section \ref{mainresult2}  is devoted to the generalization of  \cite[Theorem 3.15]{JCA2022}, that is,  the proof that the extended exterior $r(\cdot)$-sphere condition belongs to the $S$-convexity regularity class.

\section{Preliminaries} We denote by $\|\cdot\|$, $\<,\>$, $B$ and $\bar{B}$, the Euclidean norm, the usual inner product, the open unit ball and the closed unit ball, respectively. For $r>0$ and $x\in\R^n$, we set $B(x;\rho):= x + \rho B$ and $\bar{B}(x;\rho):= x + \rho \bar{B}$. The sphere centered at $x\in\R^n$ with radius $\rho>0$ is denoted by $S(x;\rho)$. For a set $A\subset\R^n$, $A^c$, $\inte A$, $\bdry A$ and $\clo A$ are the complement (with respect to $\R^n$), the interior, the boundary and the closure of $A$, respectively. The closed segment (resp. open segment)  joining two points $x$ and $y$ in $\R^n$ is denoted by $[x,y]$ (resp. $(x,y)$). For $A\subset\R^n$ and $x\in\R^n$, $[A,x]$ designates the union of all segments $[a,x]$ such that $a\in A$. For $\Omega\subset\R^n$ open and $f\colon\Omega\f(0,+\infty]$ an extended real-valued function, we denote by $\epi f$ the epigraph of $f$. The distance from a point $x$ to a set $A$ is denoted by $d_A(x)$. We also denote by $\textnormal{proj}\,_A(x)$ the set of closest points in $A$ to $x$, that is, the set of points $a$ in $A$ satisfying $d_A(x)=\|a-x\|$. For $x\in\R^n$ and $\zeta\in\R^n$ a unit vector, the directional distance from $x$ to a closed set $A$ in the direction $\zeta$, denoted by $d_A(x,\zeta)$, is defined by $$d_A(x,\zeta):=\min\{t\geq 0 : x+t\zeta\in A\},$$ where the minimum of an empty set is taken to be $+\infty$.

Now we provide some geometric definitions from proximal analysis. Our general reference for these
constructs is \cite{clsw}; see also the monographs \cite{mord,penot,rockwet,thibault}. Let $A$ be a nonempty and closed subset of $\R^n$. For $a\in A$, a vector $\zeta \in \R^n$ is said to be {\it proximal normal} to $A$ at $a$ provided that there exists $\sigma = \sigma(a,\zeta) \geq 0$ such that 
\begin{equation} \label{psi}
\langle \zeta,x-a\rangle \leq \sigma \|x-a\|^2,\;\;\forall x\in A.
\end{equation}
The relation (\ref{psi}) is commonly referred to as the {\em proximal normal inequality}. No nonzero $\zeta$ satisfying (\ref{psi}) exists if $a\in\inte A$, but this may also occur for $a\in\bdry A$. For such points, the only proximal normal is $\zeta = 0$. In view of (\ref{psi}), the set of all proximal normals to $A$ at $a$ is a convex cone, and we denote it by $N^P_A(a)$. One can easily prove that the set $\{a\in\bdry A : N_A^P(a)\not=\{0\}\}$ is dense in $\bdry A$. Now let $a\in\bdry A$, and suppose that $0\neq \zeta \in \R^n$ and $r>0$ are such that
\begin{equation} B\bigg(a+r\frac{\zeta}{\|\zeta\|};r\bigg)\cap A=\emptyset.\label{first}\end{equation} Then $\zeta$ is a proximal normal to $A$ at $a$ and we say that $\zeta$ is {\it realized by an} $r${\it-sphere}.  Note that $\zeta$ is then also realized by an $r'$-sphere for any $r'\in(0,r]$.  One can show that $\zeta$ being realized by an $r$-sphere is equivalent to the proximal normal inequality holding with $\sigma=\frac{1}{2r}$, that is,
\begin{equation} \label{psies}
\left\< \frac{\zeta}{\|\zeta\|},x-a\right\> \leq \frac{1}{2r} \|x-a\|^2,\;\;\forall x\in A,
\end{equation}
with strict inequality for $x\in \inte A$. In that case, we have \begin{equation*}\label{lastlemma} \proj_A(x)=\{a\}\;\;\,\hbox{for all}\;x\in \bigg[a,a+r\frac{\zeta}{\|\zeta\|}\bigg),\;\;\hbox{and}\;a\in\proj_A\bigg(a+r\frac{\zeta}{\|\zeta\|}\bigg).\end{equation*}
On the other hand, for $x\in A^c$, $a\in\proj_A(x)$ and $\zeta_a:=\frac{x-a}{\|x-a\|}$, we have $\zeta_a\in N_A^P(a)$, and $\zeta_a$ is realized by a $\|x-a\|$-sphere. 

We proceed to define the extended exterior $r(\cdot)$-sphere condition. For $A\subset \R^n$ a nonempty and closed set, and $r\colon\bdry A\f (0,+\infty]$ a continuous function, we say that $A$ satisfies the {\it extended exterior $r(\cdot)$-sphere condition} if for every $a\in\bdry A$, the following assertions hold:
\begin{enumerate}[$\bullet$]
\item If $a\in\bdry(\inte A)$ then there {\it exists} a unit vector $\zeta_a\in N^P_A(a)$ such that $$\begin{cases}B(a+r(a)\zeta_a;r(a))\cap A=\emptyset,&\hbox{if}\; r(a)<+\infty,\\ B(a+\rho\zeta_a;\rho)\cap A=\emptyset\;\,\hbox{for all}\;\rho>0,&\hbox{if}\; r(a)=+\infty,\end{cases}$$
\item  If $a\not \in\bdry(\inte A)$ then for {\it all} unit vectors $\zeta_a\in N^P_A(a)$, we have $$\begin{cases}B(a+r(a)\zeta_a;r(a))\cap A=\emptyset,&\hbox{if}\; r(a)<+\infty,\\ B(a+\rho\zeta_a;\rho)\cap A=\emptyset\;\,\hbox{for all}\;\rho>0,&\hbox{if}\; r(a)=+\infty.\end{cases}$$
\end{enumerate}
Since  \eqref{first} is equivalent to $\zeta$ being realized by an $r$-sphere, we deduce that  $A$ satisfies the extended exterior $r(\cdot)$-sphere condition if for every $a\in\bdry A$, the following assertions hold:
\begin{enumerate}[$\bullet$]
\item If $a\in\bdry(\inte A)$ then there exists a unit vector $\zeta_a\in N^P_A(a)$ such that $$\begin{cases}\zeta_a\;\hbox{is realized by an}\;r(a)\hbox{-sphere},&\hbox{if}\; r(a)<+\infty,\\ \zeta_a\;\hbox{is realized by a}\;\rho\hbox{-sphere}\;\hbox{for all}\;\rho>0,&\hbox{if}\; r(a)=+\infty,\end{cases}$$
\item  If $a\not \in\bdry(\inte A)$ then for all unit vectors $\zeta_a\in N^P_A(a)$, we have $$\begin{cases}\zeta_a\;\hbox{is realized by an}\;r(a)\hbox{-sphere},&\hbox{if}\; r(a)<+\infty,\\ \zeta_a\;\hbox{is realized by a}\;\rho\hbox{-sphere}\;\hbox{for all}\;\rho>0,&\hbox{if}\; r(a)=+\infty.\end{cases}$$
\end{enumerate}
On the other hand, the equivalence between \eqref{first} and \eqref{psies} yields that $A$ satisfies the extended exterior $r(\cdot)$-sphere condition if for every $a\in\bdry A$, the following assertions hold:
\begin{enumerate}[$\bullet$]
\item If $a\in\bdry(\inte A)$ then there exists a unit vector $\zeta_a\in N^P_A(a)$ such that $$\begin{cases}\left\< \zeta_a,x-a\right\> \leq \frac{1}{2r(a)} \|x-a\|^2\;\,\hbox{for all}\;x\in A,&\hbox{if}\; r(a)<+\infty,\\  \left\< \zeta_a,x-a\right\>\leq 0\;\,\hbox{for all}\;x\in A,&\hbox{if}\; r(a)=+\infty,\end{cases}$$
\item  If $a\not \in\bdry(\inte A)$ then for all unit vectors $\zeta_a\in N^P_A(a)$, we have $$\begin{cases}\left\< \zeta_a,x-a\right\> \leq \frac{1}{2r(a)} \|x-a\|^2\;\,\hbox{for all}\;x\in A,&\hbox{if}\; r(a)<+\infty,\\  \left\< \zeta_a,x-a\right\> \leq 0\;\,\hbox{for all}\;x\in A,&\hbox{if}\; r(a)=+\infty.\end{cases}$$
\end{enumerate}
Note that alternating between the above three equivalent definitions of the extended exterior $r(\cdot)$-sphere condition will be crucial for the proofs of our main results. When the set $A$ is regular closed, that is, $A=\clo(\inte A)$, then the extended exterior $r(\cdot)$-sphere condition coincides with the $\theta$-exterior sphere condition with $\theta(\cdot):=\frac{1}{2r(\cdot)}$ (where $\frac{1}{+\infty}:=0$). This latter condition, called here exterior $r(\cdot)$-sphere condition, is introduced in \cite{NA2010}, and studied in depth in the papers \cite{NA2010,DC2011,JOTA2012,JCA2022}.

\begin{remark}  One can easily see that if $A$ satisfies the extended exterior $r(\cdot)$-sphere condition, then $A$ satisfies the exterior $r(\cdot)$-sphere condition. The converse does not hold in general, see \cite[Example 2.5]{JCA2024}. On the other hand, using the proximal normal inequality, one can prove that if $A$  satisfies the extended exterior $r(\cdot)$-sphere condition, then $\clo(\inte A)$ satisfies the exterior $r(\cdot)$-sphere condition.
\end{remark}

\begin{remark} \label{nonemptiness} In this remark we prove that if $A$  satisfies the extended exterior $r(\cdot)$-sphere condition, then $N_A^P(a)\not=\{0\}$ for all $a\in\bdry A$. Indeed, let $a\in\bdry A$. If $a\in \bdry(\inte A)$, then from the definition of the extended exterior $r(\cdot)$-sphere condition, we deduce that a unit vector $\zeta_a$ belongs to $N_A^P(a)$, and hence, $N_A^P(a)\not=\{0\}$. Now we assume that $a\not\in \bdry(\inte A)$. Since the set $\{x\in\bdry A : N_A^P(x)\not=\{0\}\}$ is dense in $\bdry A$, there exists a sequence $a_n\in\bdry A$ such that $a_n\f a$, and $N_A^P(a_n)\not=\{0\}$ for all $n$. This yields, using the extended exterior $r(\cdot)$-sphere condition of $A$, that $r(a_n)\f r(a)$, and for each $n$, there exists a unit vector $\zeta_{a_n}\in N_A^P(a_n)$ such that \begin{equation} \label{nonempty} \left\< \zeta_{a_n},x-a_n\right\> \leq \frac{1}{2r(a_n)} \|x-a_n\|^2\;\,\hbox{for all}\;x\in A.\footnote{For simplicity, we assume that $\frac{1}{+\infty}=0$.}\end{equation}
Since $(\zeta_{a_n})_n$ is a sequence of unit vectors, it has a subsequence, we do not relabel, that converges  to a unit vector $\zeta_a$. Now taking $n\f \infty$ in \eqref{nonempty}, we get that $$\left\< \zeta_a,x-a\right\> \leq \frac{1}{2r(a)} \|x-a\|^2\;\,\hbox{for all}\;x\in A.$$
This gives that $\zeta_a\in N_A^P(a)$, and hence, $N_A^P(a)\not=\{0\}$. Note that in both cases, we obtained a unit vector $\zeta_a\in N_A^P(a)$ realized by an $r(a)$-sphere, if $r(a)<+\infty$, and by a $\rho$-sphere for all $\rho>0$, if $r(a)=+\infty$.
\end{remark} 

For $\Omega\subset\R^n$ an open set and for $\rho\colon\Omega\f(0,+\infty]$ a lower semicontinuous function, we say that $\Omega$ is the union of closed balls with radius function $\rho(\cdot)$ if for every $x\in\Omega$, there exists $y_x\in\Omega$ such that 
$$\begin{cases}x\in\bar{B}(y_x;\rho(x))\subset \Omega,&\hbox{if}\; \rho(x)<+\infty,\\ x\in\bar{B}(x+\delta(y_x-x);\delta)\subset \Omega\;\,\hbox{for all}\;\delta>0,&\hbox{if}\; \rho(x)=+\infty.\end{cases}$$

We terminate this section with the definition of {\it $S$-convexity}, see \cite{JCA2018,JCA2018b,JCA2022}. Let $A$ be a closed and nonempty subset of $\R^n$ and let $S$ be a set containing $A$. We say that $A$ is $S$-convex if for all $s\in S\cap A^c$ and $a\not=a'\in \bdry A$  satisfying \begin{equation}\label{ee0} \|s-a\|\leq d_{\bdry A}\bigg(a,\frac{s-a}{\|s-a\|}\bigg)\;\;\hbox{and}\;\;\|s-a'\|\leq d_{\bdry A}\bigg(a',\frac{s-a'}{\|s-a'\|}\bigg),\end{equation} we have $$ (s-a)\not\in N_A^P(a)\,\;\hbox{or}\,\;(s-a')\not\in N_A^P(a').$$

\begin{remark} The two conditions of (\ref{ee0}) are equivalent to $[a,s]\subset S$ and $[a',s]\subset S$, respectively. Hence if for $a\in\bdry A$, $\zeta\in  N_A^P(a)$ and $t\geq 0$, the segment $[a,a+t\zeta]$ is called {\it normal segment} to $A$ at $a$, then the $S$-convexity of $A$ means that no two normal segments to $A$, at two distinct points, contained in $S$, intersect in $S$. This fact is illustrated in Figure \ref{Fig1}. 
\end{remark}
\begin{figure}[t!]
\centering
\includegraphics[scale=0.86]{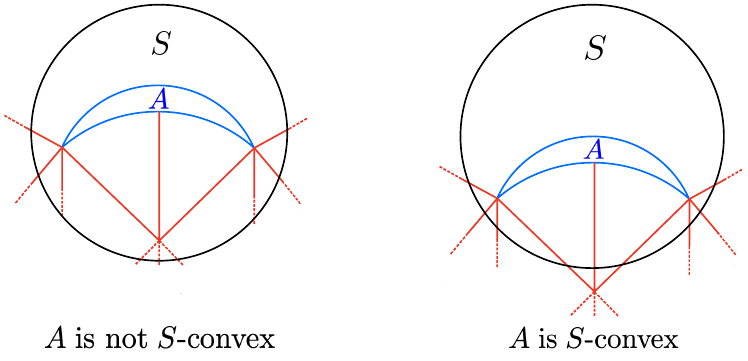}
\caption{\label{Fig1}$S$-convexity}
\end{figure}

The following basic properties of $S$-convexity are listed in \cite{JCA2018b}.

\begin{proposition}[\bbbp{\cite[Proposition 3.3]{JCA2018b}}] \label{prop2} Let $A$ and $S$ be two nonempty sets in $\R^n$ such that $A$ is closed and $A\subset S$. Then we have the following$\sp\sp:$
\begin{enumerate}[$(i)$]
\item $A$ is convex if and only if $A$ is $\R^n$-convex.
\item If $A$ is $S$-convex and $A\subset S_1\subset S$ then $A$ is $S_1$-convex.
\item If $A$ is $S$-convex and $s\in S$  with $[\proj_A(s),s]\subset S$, then $s$ has a unique projection on $A$.
\end{enumerate}
\end{proposition}

\section{Characterization of the complement as union of balls} \label{mainresult1} The main result of this section is the following theorem in which we prove that the complement of a nonempty and closed set $A$ satisfying the extended exterior $r(\cdot)$-sphere condition is nothing but the union of closed balls with the lower semicontinuous radius function $\rho(\cdot)$ defined in \eqref{radiusrho}. This generalizes \cite[Theorem 1.2]{JCA2024} to the variable radius case, as highlighted in the introduction, see Corollary \ref{coro1}.

\begin{theorem}\label{th1} Let $A\subset\R^n$ be a nonempty and closed set satisfying the extended exterior $r(\cdot)$-sphere condition. Then $A^c$ is the union of closed balls with radius function $\rho(\cdot)$, where $\rho\colon A^c\f(0,+\infty]$ is the lower semicontinuous function defined by \begin{equation*} \label{radiusrhobis} \rho(x):=\min\left\{\frac{r(a)}{2} : a\in\proj_{A}(x)\right\}.\end{equation*}
\end{theorem}

When $r(\cdot)=r>0$ is a constant, Theorem \ref{th1} leads to the following corollary that coincides with \cite[Theorem 1.2]{JCA2024}.

\begin{corollary}[{\bbbp\cite[Theorem 1.2]{JCA2024}}] \label{coro1} Let  $A\subset\R^n$ be a nonempty and closed set satisfying the extended exterior $r$-sphere condition for some $r>0$. Then $A^c$ is the union of closed balls with common radius $\frac{r}{2}$.
\end{corollary}

Now we provide the proof of Theorem \ref{th1}.\vspace{0.2cm}\\
{\bf Proof of  Theorem \ref{th1}.} We begin by proving the lower semicontinuity of $\rho(\cdot)$. First we note that the infimum defining $\rho(x)$, for each $x\in A^c$, is attained since $\proj_{A}(x)$ is compact and $r(\cdot)$ is continuous on $\bdry A$. Let $(x_n,r_n)\in\epi \rho$ such that $(x_n,r_n)\f (\x,\r)$ as $n\f+\infty$. For each $n$, there exists $a_n\in\proj_{A}(x_n)$ such that 
\begin{equation}\label{randrho} \frac{r(a_n)}{2}=\rho(x_n) \leq r_n.\end{equation}
Having $x_n\f\x$ and $a_n\in\bdry A$, we deduce that the sequence $a_n$ is bounded, and hence, we can assume that $a_n\f \a\in\bdry A$. Since the the projection map $\proj_{A}(\cdot)$ is closed, we deduce that $\a\in \proj_{A}(\x)$. Then, using \eqref{randrho}, we get that $$\rho(\x)\leq \frac{r(\a)}{2}=\lim_{n\f+\infty} \frac{r(a_n)}{2}\leq \lim_{n\f+\infty} r_n=\r.$$
Therefore, $(\x,\r)\in\epi \rho$. This terminates the proof of the lower semicontinuity of $\rho(\cdot)$. 

Now we prove that $A^c$ is the union of closed balls with radius function $\rho(\cdot)$. Let $x\in A^c$. There are two cases to consider.\vspace{0.1cm}\\
\underline{Case 1:} $\rho(x)<+\infty$.\vspace{0.1cm}\\
Let $a_x\in \bdry A$ such that $\rho(x)=\frac{r(a_x)}{2}$. Since $a_x\in\proj_{A}(x)$, we have, for $\rho_x:=\|x-a_x\|>0$, that \begin{equation} \label{case0} B(x;\rho_x)\cap A=\emptyset.\end{equation}
If $\rho(x)<\rho_x$ then from \eqref{case0} we get, for $y_x:=x$, that $$x\in\bar{B}(y_x;\rho(x))\subset A^c.$$
Now we assume that \begin{equation}\label{wedrot}  \frac{r(a_x)}{2}=\rho(x)\geq\rho_x.\end{equation}
\underline{Case 1.1:} $a_x\not\in\bdry (\inte A)$.\vspace{0.1cm}\\
Since $A$ satisfies the extended exterior $r(\cdot)$-sphere condition, we conclude that  the unit vector $\zeta_{\ax}:=\frac{x-a_x}{\|x-a_x\|}\in N_A^P(a_x)$ is realized by an  $r(a_x)$-sphere. That is, \begin{equation} \label{inclusionballs1} B(a_x+r(a_x)\zeta_{\ax};r(a_x))\subset A^c.
\end{equation}
On the other hand, one can easily prove using \eqref{wedrot} that 
$$\bar{B}\bigg(x+\frac{r(a_x)}{2}\zeta_{\ax};\frac{r(a_x)}{2}\bigg)\subset B(a_x+r(a_x)\zeta_{\ax};r(a_x)).$$
Now, this latter inclusion with the inclusion of \eqref{inclusionballs1} yield, for $y_x:=x+\rho(x)\zeta_{\ax}$, that 
$$x\in\bar{B}(y_x;\rho(x))= \bar{B}\bigg(x+\rho(x)\zeta_{\ax};\rho(x)\bigg)=\bar{B}\bigg(x+\frac{r(a_x)}{2}\zeta_{\ax};\frac{r(a_x)}{2}\bigg)\subset A^c.$$
\underline{Case 1.2:} $a_x\in\bdry (\inte A)$.\vspace{0.1cm}\\
Then for every $\e\in(0,\rho_x)$, we have $B(a_x;\e)\cap \inte A\not=\emptyset.$ Let $z_{\e}\in B(a_x;\e)\cap \inte A$, and denote by $\xi_{\e}:=\frac{z_\e-x}{\|z_\e-x\|}$.\vspace{0.2cm}\\
\underline{Claim:} $(x,z_\e)\cap (\bdry A)\cap B(a_x;\e)\not=\emptyset$.\vspace{0.1cm}\\
To prove this claim, we first prove that the line $L:=\{x_t:=x+t\xi_\e : t\in\R\}$ intersects the sphere $S(a_x;\e)$ at two points $x_{t_1}$ and $x_{t_2}$. We consider the function $$f(t):=\|x_t-a_x\|^2-\e^2,\,\;t\in\R.$$
For $t_\e:=\|z_\e-x\|$, we have, since $z_\e\in  B(a_x,\e)$, that $f(t_\e)<0$. Replacing $x_t$ by $x+t\xi_\e$ in the formula of $f(t)$, we obtain, for $\zeta_{\ax}:=\frac{x-a_x}{\|x-a_x\|}$, that $$f(t)=t^2+2t\rho_x\<\zeta_{a_x},\xi_\e\>+(\rho_x^2-\e^2),\;\,t\in\R.$$
Having $f(t_\e)<0$, we deduce that the discriminant $\Delta'$ of $f(t)$ is positive, that is, $$\Delta'=(\rho_x\<\zeta_{a_x},\xi_\e\>)^2-(\rho_x^2-\e^2)>0.$$ 
Then, the equation $f(t)=0$ has two distinct real roots $t_1<t_2$, and $f(t)<0$ if and only if $t\in (t_1,t_2)$. This yields that $$x_{t_1}\in S(a_x;\e),\;x_{t_2}\in S(a_x;\e),\;\hbox{and}\;x_t\in B(a_x;\e)\Longleftrightarrow t\in(t_1,t_2).$$
Since $f(t_\e)<0$, $t_\e>0$, $t_1t_2=\rho_x^2-\e^2>0$, we conclude that $0<t_1<t_{\e}<t_2$. Hence, $$t_1= -\rho_x\<\zeta_{a_x},\xi_\e\>-\sqrt{\Delta'}\;\;\hbox{and}\;\, t_2= -\rho_x\<\zeta_{a_x},\xi_\e\>+\sqrt{\Delta'}.$$
This yields, using that $t_1>0$ and $|\<\zeta_{a_x},\xi_\e\>|\leq 1$, that $-\<\zeta_{a_x},\xi_\e\>\in(0,1]$. Then, $t_1<\rho_x$. Since $\|x_{t_1}-x\|=t_1$, we conclude that $$\|x_{t_1}-x\|<\rho_x,\;\hbox{which gives, using \eqref{case0}, that}\;x_{t_1}\in A^c.$$
Having $x_{t_\e}=z_{\e}\in \inte A$ and $x_{t_1}\in A^c$, we deduce the existence of $t^{\e}\in(t_1,t_{\e})$ such that $x_{t^{\e}}\in\bdry A$. Taking $a_\e:=x_{t^{\e}}=x+t^{\e}\xi_\e\in(x,z_\e)\cap\bdry A$, we get, since $f(t^{\e})<0$, that $a_\e\in B(a_x;\e)$. Hence, \begin{equation}\label{aepsilon} a_\e\in (x,z_\e)\cap (\bdry A)\cap B(a_x;\e).\end{equation}
This terminates the proof of the claim. By Remark \ref{nonemptiness}, there exists  a unit vector $\zeta_\e\in N_S^P(a_\e)$ realized by an $r(a_\e)$-sphere. Hence, for $y_\e:=a_\e+r(a_\e){\zeta_\e}$, we have \begin{equation}\label{case3.1} B(y_\e;r(a_\e))\cap A=\emptyset.\end{equation}
Now, \eqref{aepsilon} with the continuity of $r(\cdot)$, yield that \begin{equation}\label{raetorax} \lim_{\e\f 0} r(a_\e)=r(a_x).\end{equation}
Therefore, taking $\e$ sufficiently small, we can assume, using \eqref{wedrot}, that  \begin{equation} \label{revsrax} r(a_\e)>\frac{r(a_x)}{2}=\rho(x)\geq \rho_x.\end{equation}
Note that the first inequality of \eqref{revsrax} yields that $r(a_x)-r(a_\e)<r(a_\e)$.\vspace{0.1cm}\\
 \underline{Case 1.2.1:} $0\leq \|y_\e-x\|\leq r(a_x)-r(a_\e)<r(a_\e)$.\vspace{0.1cm}\\
Then using \eqref{revsrax}, we get that  $\|y_\e-x\|< \frac{r(a_x)}{2}=\rho(x)$. Now taking $y_x:=y_\e$, and using \eqref{case3.1} and \eqref{revsrax}, we deduce that $$x\in \bar{B}(y_x;\rho(x))= \bar{B}(y_\e;\rho(x))\subset B(y_\e;r(a_\e))\subset A^c.$$
\underline{Case 1.2.2:} $r(a_x)-r(a_\e)<\|y_\e-x\|<r(a_\e)$.\vspace{0.1cm}\\
If $\|y_\e-x\|=0$, then for $y_x:=x=y_\e$ we have, using \eqref{case3.1} and \eqref{revsrax}, that $$x\in\bar{B}(y_x;\rho(x))= \bar{B}\bigg(y_\e;\frac{r(a_x)}{2}\bigg)\subset B(y_\e;r(a_\e))\subset A^c.$$
If $\|y_\e-x\|\not=0$, then we define $y_x:=x+\rho(x)\frac{y_\e-x}{\|y_\e-x\|}$. For $v\in\bar{B}(y_x;\rho(x))$, we have 
$$\|v-y_\e\|\leq \|v-y_x\|+\|y_x-y_\e\|\leq \rho(x) +|\rho(x)-\|y_\e-x\||<r(a_\e).$$
This yields, using \eqref{case3.1}, that $$x\in \bar{B}(y_x;\rho(x))\subset B(y_\e;r(a_\e))\subset A^c.$$
\underline{Case 1.2.3:} $\|y_\e-x\|\geq r(a_\e)$.\vspace{0.1cm}\\
We denote by \begin{equation} \label{rhoepsilon} \rho_\e:=\frac{\rho_x^2\|y_\e-x\|}{\|y_\e-x\|^2+\rho_x^2-r(a_\e)^2}>0.\end{equation}
We claim that $$ B\bigg(x+\rho_\e\frac{y_\e-x}{\|y_\e-x\|}; \rho_\e \bigg) \subset B(x;\rho_x)\cup B(y_\e;r(a_\e)).$$
Indeed, let $v\in B\big(x+\rho_\e\frac{y_\e-x}{\|y_\e-x\|}; \rho_\e \big)$ and assume that $v\not\in B(x;\rho_x)$. 
We have \begin{eqnarray*}\|v-y_\e\|^2 &=& \|(v-x)-(y_\e-x)\|^2\\&=&  \|v-x\|^2+\|y_\e-x\|^2-2\<v-x,y_\e-x\> \\&<& \|v-x\|^2+\|y_\e-x\|^2 - \frac{\|y_\e-x\|}{\rho_\e}\|v-x\|^2\\&=& \|y_\e-x\|^2 + \|v-x\|^2\left(1-\frac{\|y_\e-x\|}{\rho_\e}\right)\\&=&  \|y_\e-x\|^2  + \|v-x\|^2\left(1-\frac{\|y_\e-x\|^2+\rho_x^2-r(a_\e)^2}{\rho_x^2}\right) \\&=&  \|y_\e-x\|^2  + \|v-x\|^2\left(\frac{r(a_\e)^2-\|y_\e-x\|^2}{\rho_x^2}\right)\\&\leq& \|y_\e-x\|^2 +r(a_\e)^2-\|y_\e-x\|^2=r(a_\e)^2.\end{eqnarray*}
Hence, $v\in B(y_\e;r(a_\e))$. Therefore \begin{equation}\label{ballsrelation} B\bigg(x+\rho_\e\frac{y_\e-x}{\|y_\e-x\|}; \rho_\e \bigg) \subset B(x;\rho_x)\cup B(y_\e;r(a_\e)). \end{equation}
On the other hand, using that $a_\e=x+t^\e\xi_\e$ with $t^\e<t_\e$, $\zeta_\e\in N_A^P(a_\e)$ is unit and realized by an $r(a_\e)$-sphere, and having $z_\e\in\inte A$, we  deduce that
\begin{eqnarray}  \nonumber \|y_\e-x\|^2 =  \|t^\e\xi_\e+r(a_\e)\zeta_\e\|^2   &=& (t^\e)^2+r(a_\e)^2+2t^\e r(a_\e)\<\zeta_\e,\xi_\e\> \\ \nonumber &<&(t^\e)^2+r(a_\e)^2+t^\e \|z_\e-a_e\|\\ \nonumber &=& (t^\e)^2+ r(a_\e)^2 +t^\e(t_\e-t^\e) \\ \nonumber &=& r(a_\e)^2 +t^\e t_\e\\ \nonumber &<&  r(a_\e)^2 + t_\e^2\\ \nonumber &=&  r(a_\e)^2 + \|z_\e-x\|^2\\ \nonumber &\leq&  r(a_\e)^2 + (\|z_\e-a_x\|+\|a_x-x\|)^2  \\ \nonumber &\leq&  r(a_\e)^2 + (\e+\rho_x)^2 \\ &=& \label{zubad} r(a_\e)^2  + \rho_x^2 + \e^2+ 2\e\rho_x.\end{eqnarray}
Furthermore, one can easily prove, using \eqref{revsrax}, \eqref{rhoepsilon}, and the inequality $\|y_\e-x\|\geq r(a_\e)$, that \begin{equation} \label{rhoepsilonbigger}\rho_\e>\frac{r(a_\e)}{2}\Longleftrightarrow \sqrt{r(a_\e)^2-\rho_x^2} \leq \|y_\e-x\| < \frac{\rho_x^2+\sqrt{\Delta'}}{r(a_\e)} \Longleftrightarrow  \|y_\e-x\|< \frac{\rho_x^2+\sqrt{\Delta'}}{r(a_\e)},\end{equation}
where $\Delta':=\rho_x^4+r(a_\e)^4-r(a_\e)^2\rho_x^2=(\rho_x^2-r(a_\e)^2)^2+\rho_x^2r(a_\e)^2>0.$
Now, combining \eqref{zubad}  with the inequality \begin{eqnarray*} \bigg(\frac{\rho_x^2+\sqrt{\Delta'}}{r(a_\e)}\bigg)^2 &=&\frac{\rho_x^4}{r(a_\e)^2}+\frac{\rho_x^4}{r(a_\e)^2}+r(a_\e)^2 -\rho_x^2+ \frac{2\rho_x^2}{r(a_\e)^2}\sqrt{\rho_x^4+r(a_\e)^4-r(a_\e)^2\rho_x^2} \\&=&\frac{2\rho_x^4}{r(a_\e)^2}+r(a_\e)^2 -\rho_x^2+ \frac{2\rho_x^2}{r(a_\e)^2}\sqrt{\bigg(r(a_\e)^2-\frac{1}{2}\rho_x^2\bigg)^2+\frac{3}{4}\rho_x^4} \\& >& \frac{2\rho_x^4}{r(a_\e)^2}+r(a_\e)^2 -\rho_x^2+ \frac{2\rho_x^2}{r(a_\e)^2}\bigg(r(a_\e)^2-\frac{1}{2}\rho_x^2\bigg)  \\&=& r(a_\e)^2+\rho_x^2+\frac{\rho_x^4}{r(a_\e)^2},\end{eqnarray*} we conclude that \begin{equation} \label{toz0} \|y_\e-x\|^2 -  \Bigg(\frac{\rho_x^2+\sqrt{\Delta'}}{r(a_\e)}\Bigg)^2<\e^2+2\e\rho_x-\frac{\rho_x^4}{r(a_\e)^2}.\end{equation}
This yields, using  \eqref{raetorax}, that \begin{equation}\label{ouff} \lim_{\e\f0}\left[\|y_\e-x\|^2 -  \Bigg(\frac{\rho_x^2+\sqrt{\Delta'}}{r(a_\e)}\Bigg)^2\right]\leq -\frac{\rho_x^4}{r(a_x)^2}<0.\end{equation}
Then, for $\e$ sufficiently small, we have \begin{equation*} \|y_\e-x\|^2 -  \Bigg(\frac{\rho_x^2+\sqrt{\Delta'}}{r(a_\e)}\Bigg)^2<0,\end{equation*} and hence, using \eqref{rhoepsilonbigger}, $\rho_\e>\frac{r(a_\e)}{2}$. We claim that $\e$ can be taken sufficiently small so that  $\rho_\e>\frac{r(a_x)}{2}$. Indeed, if not then there exists a decreasing and positive sequence $(\e_n)_n$ such that $\e_n\f0$ and 
 $$ \lim_{n\f+\infty} \bigg(\rho_{\en}-\frac{r(a_{\en})}{2}\bigg)=0.$$ This yields that \begin{equation}\label{secondorder} r(a_x)\|\y-x\|^2-2\rho_x^2\|\y-x\|+r(a_x)(\rho_x^2-r(a_x)^2)=0,\end{equation}
where $\y$ is the limit as $n\f+\infty$ of $y_{\en}$. Note that this latter convergence can be assumed since $a_{\en} \f a_x$, $r(a_{\en}) \f r(a_x)$ and $\zeta_{\en}$ is unit. Solving the second order equation \eqref{secondorder} and using \eqref{revsrax}, we deduce that \begin{equation}\label{toz} \|\y-x\|= \frac{\rho_x^2+\sqrt{\rho_x^4+r(a_x)^4-r(a_x)^2\rho_x^2}}{r(a_x)}.\end{equation}
Now, taking $n\f+\infty$ in \eqref{toz0} (after replacing $\e$ by $\e_n$), and using \eqref{toz}, we find that $$0= \|\y-x\|^2- \Bigg(\frac{\rho_x^2+\sqrt{\rho_x^4+r(a_x)^4-r(a_x)^2\rho_x^2}}{r(a_x)}\bigg)^2\leq -\frac{\rho_x^4}{r(a_x)^2}<0, $$
which gives the desired contradiction. Therefore, $\e$ can be taken small enough so that  $$\rho_\e>\frac{r(a_x)}{2}.$$ The latter inequality with \eqref{ballsrelation}, \eqref{case0}, and \eqref{case3.1}, yield that for $y_x:=x+\rho(x)\frac{y_\e-x}{\|y_\e-x\|}$, we have \begin{eqnarray*}  x\in\bar{B}\bigg(x+\rho(x)\frac{y_\e-x}{\|y_\e-x\|}; \rho(x)\bigg)&=&\bar{B}\bigg(x+\frac{r(a_x)}{2}\frac{y_\e-x}{\|y_\e-x\|}; \frac{r(a_x)}{2}\bigg)\\ &\subset&   {B}\bigg(x+\rho_\e\frac{y_\e-x}{\|y_\e-x\|}; \rho_\e\bigg)\cup\{x\}\\ &\subset&  B(x;\rho_x)\cup B(y_\e;r(a_\e))\subset A^c. \end{eqnarray*}
\underline{Case 2:} $\rho(x)=+\infty$.\vspace{0.1cm}\\
Then for all $a\in\proj_{A}(x)$, we have $r(a)=+\infty$. We fix $a_x\in\proj_A(x)$. Then we have $r(a_x)=+\infty$. Moreover, for $\rho_x:=\|x-a_x\|>0$, we have $B(x;\rho_x)\cap A=\emptyset.$ \vspace{0.1cm}\\
\underline{Case 2.1:} $a_x\not\in\bdry (\inte A)$.\vspace{0.1cm}\\
Since $A$ satisfies the extended exterior $r(\cdot)$-sphere condition and $r(a_x)=+\infty$, we conclude that  the unit vector $\zeta_{\ax}:=\frac{x-a_x}{\|x-a_x\|}\in N_A^P(a_x)$ is realized by a  $\delta$-sphere for all $\delta>0$. That is, $$B(a_x+\delta\zeta_{\ax};\delta)\subset A^c,\;\,\forall\delta>0.$$
Hence, for $y_x:=x+\zeta_{\ax}$, we have that for all $\delta>0$, 
$$\bar{B}(x+\delta(y_x-x);\delta)=\bar{B}(a_x+(\delta+\|x-a_x\|)\zeta_{\ax};\delta) \subset B(a_x+(\delta+\|x-a_x\|)\zeta_{\ax};\delta+\|x-a_x\|)\subset A^c.$$
\underline{Case 2.2:} $a_x\in\bdry (\inte A)$.\vspace{0.1cm}\\
As in the Case 1.2 above, we define, for every $\e\in(0,\rho_x)$, the points $z_\e$ and $a_\e$, and the vector $\xi_\e$. The continuity of $r(\cdot)$ yields that $$\lim_{\e\f 0} r(a_\e)=r(a_x)=+\infty.$$
\underline{Case 2.2.1:} $\exists\e_0>0$ such that for all $\e\geq\e_0$ we have $r(a_\e)=+\infty$.\vspace{0.1cm}\\
Let $n\geq1$ satisfying $n\geq\rho_x$. Using arguments similar to those used in Case 1.2 with $r(a_\e)$ and $r(a_x)$ both replaced by $2n$, we obtain the existence of unit vector $\omega_n$ such that  $$\begin{cases} x\in\bar{B}(x;n)\subset A^c&\hbox{if}\;y_\e=x,\\
x\in\bar{B}(x+n\omega_n; n)\subset A^c&\hbox{if}\;y_\e\not=x.\end{cases}$$
In both cases, there exists a unit vector $\omega_n$ such that \begin{equation} \label{incluimp} x\in\bar{B}\bigg(x+\frac{n}{2}\omega_n;\frac{n}{2}\bigg)\subset A^c,\;\,\forall n\geq \rho_x.\end{equation}
Since the vector $\omega_n$ is unit, we can assume that the sequence $(\omega_n)_n$ converges as $n\f +\infty$ to a unit vector $\omega_0$. For $y_x:=x+\omega_0$, we claim that $\bar{B}(x+\delta(y_x-x);\delta)\subset A^c$ for all $\delta>0$. Indeed, let $\delta>0$ and let $v\in B(x+\delta(y_x-x);\delta)$. If $v\not\in\bar{B}\big(x+\frac{n}{2}\omega_n;\frac{n}{2}\big)$ for all $n\geq \rho_x$, then $$\left\|v-x-\frac{n}{2}\omega_n\right\|>\frac{n}{2},\;\,\forall n\geq \rho_x.$$ 
Hence, for $n$ sufficiently large, we have \begin{eqnarray*}\frac{n}{2}&<& \left\|v-x-\delta\omega_0+ \delta\omega_0-\delta \omega_n+\delta\omega_n -\frac{n}{2}\omega_n\right\| \\&\leq&  \left\|v-x-\delta\omega_0 \right\|+ \delta\left\| \omega_0-  \omega_n\right\|  +  \Big(\frac{n}{2}-\delta\Big).  \end{eqnarray*}
This yields that for $n$ sufficiently large, $$0< (\left\|v-x-\delta\omega_0 \right\|-\delta) + \delta\left\| \omega_0-  \omega_n\right\|.$$
Taking $n\f+\infty$ in this latter inequality, we obtain that $$\delta\leq \left\|v-x-\delta\omega_0 \right\|=\left\|v-x-\delta(y_x-x) \right\|, $$ which gives the desired contradiction. Hence, there exists $n_0=n(\delta,v)\geq \rho_x$ such that $$v\in \bar{B}\bigg(x+\frac{n_0}{2}\omega_{n_0};\frac{n_0}{2}\bigg)\subset A^c,$$ where the last inclusion follows from \eqref{incluimp}. Therefore, $B(x+\delta(y_x-x);\delta)\subset A^c$ for all $\delta>0$. Now to prove that $\bar{B}(x+\delta(y_x-x);\delta)\subset A^c$  for all $\delta>0$, it is sufficient to remark that $$\bar{B}(x+\delta(y_x-x);\delta)\subset B (x+2\delta(y_x-x);2\delta)\cup\{x\}\subset A^c,\;\,\forall \delta>0.$$ 
\underline{Case 2.2.2:} $\exists (\e_n)_n$ such that $\e_n>0$ for all $n$, $\e_n\bp\searrow$\sp,  $\e_n\f 0$, and $r(a_{\en})<+\infty$ for all $n$.\vspace{0.1cm}\\
Since $r(a_{\e_n})\f +\infty$, we can assume that along a subsequence, we do not relabel,  we have $$r(a_{\e_n})\geq2\rho_x\;\hbox{for all}\;n,\;r(a_{\e_n})\nearrow,\;\hbox{and}\; r(a_{\e_n})\f +\infty.$$
Using arguments similar to those used in Case 1.2 with $r(a_x)$ replaced by $r(a_{\e_n})$, we obtain that
\begin{equation*} \label{last} \begin{cases} x\in\bar{B}\Big(x;\frac{r(a_{\en})}{2}\Big)\subset A^c&\hbox{if}\;y_{\e_n}=x,\vspace{0.1cm}\\ 
x\in\bar{B}\Big(x+\frac{r(a_{\en})}{2}\frac{y_{\en}-x}{\|y_{\en}-x\|}; \frac{r(a_{\en})}{2}\Big)\subset A^c&\hbox{if}\;y_{\en}\not=x\,\;\hbox{and}\,\;\|y_{\en}-x\| < r(a_{\en}),\vspace{0.1cm}\\ 
B\Big(x+\rho_{\en}\frac{y_{\en}-x}{\|y_{\en}-x\|}; \rho_{\en} \Big)\subset A^c, &\hbox{if}\;\|y_{\en}-x\|\geq r(a_{\en}),
\end{cases}\end{equation*}
where $\rho_{\en}$ is defined in \eqref{rhoepsilon}. This yields the existence of sequence $(\omega_n)_n$ of unit vectors satisfying 
\begin{equation*} \label{last2} \begin{cases}
x\in\bar{B}\Big(x+\frac{r(a_{\en})}{4}\omega_n; \frac{r(a_{\en})}{4}\Big)\subset A^c&\hbox{if}\;\|y_{\en}-x\| < r(a_{\en}),\vspace{0.1cm}\\ 
B\Big(x+\rho_{\en}\omega_n; \rho_{\en} \Big)\subset A^c, &\hbox{if}\;\|y_{\en}-x\|\geq r(a_{\en}).
\end{cases}\vspace{0.1cm}\end{equation*}
\underline{Case 2.2.2.1:} There are infinite indices $n$, we do not relabel, for which $\|y_{\en}-x\| < r(a_{\en})$.\vspace{0.1cm}\\
We proceed as in Case 2.2.1 where we replace the radius $\frac{n}{2}$ by $\frac{r(a_{\en})}{4}$. We obtain the existence of $y_x$ such that $$\bar{B}(x+\delta(y_x-x);\delta)\subset A^c,\;\,\forall \delta>0.$$ 
\underline{Case 2.2.2.2:} There are infinite indices $n$, we do not relabel, for which $\|y_{\en}-x\|\geq r(a_{\en})$.\vspace{0.1cm}\\
We claim that $ \rho_{\en}\f+\infty$. Indeed, by \eqref{zubad} and since $\|y_{\en}-x\|\geq r(a_{\en})$, we have that $$\lim_{n\f+\infty}\|y_{\en}-x\|=+\infty\;\;\hbox{and}\;\;\rho_x^2\leq \lim_{n\f+\infty}\big(\|y_{\en}-x\|^2+\rho_x^2-r(a_{\en})^2\big)\leq 2\rho_x^2.$$
This yields that $$\lim_{n\f +\infty} \rho_{\en}=+\infty.\footnote{The limit in \eqref{ouff} is now $\leq 0$, and hence, we cannot deduce that $\rho_{\en}>\frac{r(a_{\en})}{2}$ for $n$ sufficiently large.}$$
Hence, we can assume that along a subsequence, we do not relabel,  we have $$\rho_{\e_n}\nearrow\;\hbox{and}\; \rho_{\e_n}\f +\infty.$$
Now, proceeding as Case 2.2.1 where we replace the radius $\frac{n}{2}$ by $\rho_{\e_n}$, and the closed ball $\bar{B}(x+\rho_{\e_n}\omega_n;\rho_{\e_n})$ by the open ball ${B}(x+\rho_{\e_n}\omega_n;\rho_{\e_n})$, we obtain the existence of $y_x$ such that $$\bar{B}(x+\delta(y_x-x);\delta)\subset A^c,\;\,\forall \delta>0.$$ 

The proof of  Theorem \ref{th1} is terminated. \hfill $\blacksquare$
 
 \begin{remark} In Theorem \ref{th1}, the function $\rho(\cdot)$ is not necessarily continuous. Indeed, for $A:=\{(x,y)\in\R^2 : y\geq 2\;\hbox{or}\;y\leq 0\}$, and $r(x):=1$ on the line $y=2$ and $r(x):=\frac{1}{2}$ on  the line $y=0$, we have $A$ satisfies the extended exterior $r(\cdot)$-sphere condition, with 
$$\rho(0,1)= \frac{1}{4},\,\;\hbox{and}\,\;\rho\bigg(0,1+\frac{1}{n}\bigg)= \frac{1}{2}\,\;\,\hbox{for all}\,\; n\geq 2.$$
This yields that $\rho(\cdot)$ is not continuous on $(0,1)$. 
\end{remark}

\section{Characterization using \texorpdfstring{$S$-convexity}{Lg}} \label{mainresult2} Before stating the main result of this section, we introduce some notations and definitons. For $A\subset\R^n$ a nonempty and closed set, $a\in\bdry A$, $\zeta_a\in N_A^P(a)$ unit, and $r\colon\bdry A\f (0,\infty]$ a continuous function, we define: 
\begin{enumerate}[$\bullet$]\setlength\itemsep{0.3em}
\item $\rho(a,\zeta_a):=\max\{\rho : \zeta_a\;\hbox{is realized by a}\;\rho\hbox{-sphere}\}$.
\item $\rho(a,\zeta_a,r):=\max\{\rho : \rho\leq r(a)\;\hbox{and}\;\zeta_a\;\hbox{is realized by a}\;\rho\hbox{-sphere}\}$.
\item $A^{\up}:=\Big\{x\in A^c : \proj_A(x)=\{a\}\;\hbox{and}\;\|x-a\|<\rho\Big(a,\frac{x-a}{\|x-a\|}\Big)\Big\}.$
\item $A_{\up\sr}:=\Big\{x\in A^c : \proj_A(x)=\{a\}\;\hbox{and}\;\|x-a\|<\rho\Big(a,\frac{x-a}{\|x-a\|},r\Big)\Big\}.$
\item $\bdry_{\rb} (\inte A):=\{a\in\bdry (\inte A) : \exists\zeta\in N_A^P(a)\;\hbox{unit with}\;\rho(a,\zeta)\geq r(a)\}.$
\end{enumerate}
In the following proposition, we present some fundamental properties of the constructs introduced above. These properties are extracted from \cite[Section 3]{JCA2022}.
\begin{proposition} \label{propfun} Let $A\subset\R^n$ a nonempty and closed set, $a\in\bdry A$, $\zeta_a\in N_A^P(a)$ {\it unit}, and $r\colon\bdry A\f (0,\infty]$ a continuous function. Then the following assertions hold$\sp:$
\begin{enumerate}[$(i)$]
\item $\rho(a,\zeta_a)\in (0,+\infty]$, and $$\begin{cases} 0<\rho(a,\zeta_a,r)=\rho(a,\zeta_a)\leq +\infty, &\hbox{if}\;\,r(a)=+\infty,\\  0<\rho(a,\zeta_a,r)=\min\{\rho(a,\zeta_a),r(a)\}, &\hbox{if}\,\;r(a)<+\infty. \end{cases} $$
\item $A_{\up\sr}\subset A^{\up}$. Moreover, $$A^{\up}=\bigcup_{\substack{a\in\bdry A\\\zeta\in N_A^P(a)\,\textnormal{unit}}} (a,a+\rho(a,\zeta)\zeta)\;\,\hbox{and}\;\,A_{\up\sr}=\bigcup_{\substack{a\in\bdry A\\\zeta\in N_A^P(a)\,\textnormal{unit}}} (a,a+\rho(a,\zeta,r),\zeta).$$
\item $A$ is $A\cup A^{\up}$-convex.
\item $\bdry_{\rb} (\inte A)$ is closed.
\end{enumerate}
\end{proposition}

\begin{remark} \label{newrem0}  From the definition of $\rho(a,\zeta_a)$,  one can easily deduce that a nonempty and closed set $A\subset \R^n$ satisfies the extended exterior $r(\cdot)$-sphere condition if for every $a\in\bdry A$, the following assertions hold:
\begin{enumerate}[$\bullet$]

\item If $a\in\bdry(\inte A)$ then there exists a unit vector $\zeta_a\in N^P_A(a)$ such that $\rho(a,\zeta_a)\geq r(a)$.
\item  If $a\not\in\bdry(\inte A)$ then for all unit vectors $\zeta_a\in N^P_A(a)$ we have $\rho(a,\zeta_a)\geq r(a)$.
\end{enumerate}
\end{remark}

\begin{remark} \label{newrem1} In this remark, we prove that for $A\subset \R^n$ nonempty and closed, and for $\rho\colon A^c\f(0,+\infty]$ continuous, if $A$ does not satisfy the extended exterior $r(\cdot)$-sphere condition then $$\begin{cases} \exists a\in\bdry(\inte A)\;\hbox{such that}\;N_A^P(a)\not=\{0\},\;\hbox{and}\;\rho(a,\zeta)<r(a),\;\forall\zeta\in N_A^P(a)\;\hbox{unit},\\ \hbox{or} \\  \exists a\in(\bdry A)\cap (\bdry(\inte A))^c\;\hbox{and}\;\exists \zeta_a\in  N_A^P(a)\;\hbox{unit such that}\; \rho(a,\zeta_a)<r(a).\end{cases}$$
Indeed, if $A$ does not satisfy the extended exterior $r(\cdot)$-sphere condition, then by Remark \ref{newrem0}, we have $$\begin{cases} \exists a\in\bdry(\inte A)\;\hbox{such that}\;\rho(a,\zeta)<r(a),\;\forall\zeta\in N_A^P(a)\;\hbox{unit},\\ \hbox{or} \\  \exists a\in(\bdry A)\cap (\bdry(\inte A))^c\;\hbox{and}\;\exists \zeta_a\in  N_A^P(a)\;\hbox{unit such that}\; \rho(a,\zeta_a)<r(a).\end{cases}$$
We assume that \begin{equation}\label{lastream1} \forall a\in(\bdry A)\cap (\bdry(\inte A))^c\;\hbox{and}\;\forall \zeta_a\in  N_A^P(a),\;\hbox{we have}\; \rho(a,\zeta_a)\geq r(a). \end{equation}
Then, there exists $b\in\bdry(\inte A)$ such that $\rho(b,\zeta)<r(b)$ for all $\zeta\in N_A^P(b)$ unit. Since the set of points $x\in\bdry A$ for which $N_A^P(x)\not=\{0\}$ is dense in $\bdry A$, we deduce the existence of a sequence $b_n\in\bdry A$ such that $b_n\f b$ and $N_A^P(b_n)\not=\{0\}$. For all $n$, we fix $\zeta_n\in N_A^P(b_n)$ unit. Since $\zeta_n$ is unit, we can assume that $\zeta_n\f \zeta_b$ unit. \vspace{0.1cm}\\
\underline{Claim 1:}  $\exists N\in\N$ such that for all $n\geq N$, we have $b_n\in\bdry (\inte A)$. \vspace{0.1cm}\\ 
Indeed, if not then $(b_n)_n$ has a subsequence, we do not relabel, such that $b_n\in (\bdry(\inte A))^c$ for all $n$. This yields, using \eqref{lastream1}, that $\rho(b_n,\zeta_n)\geq r(b_n)$ for all $n$. Taking $n\f+\infty$ in the proximal normal inequality of $\zeta_n$, and using the continuity of $r(\cdot)$, we deduce that $$\zeta_b\in N_A^P(b)\;\,\hbox{and}\;\,\rho(b,\zeta_b)\geq r(b),$$
which gives the desired contradiction. \vspace{0.1cm}\\
\underline{Claim 2:} $\exists \bar{N}\geq N$ such that $\rho(b_{n},\zeta)<r(b_{n})$ for all $n\geq \bar{N}$ and for all $\zeta\in N_A^P(b_n)$ unit.\vspace{0.1cm}\\
If not, then $(b_n)_n$ has a subsequence, we do not relabel, such that for all $n$, one can find $\xi_n\in N_A^P(b_n)$ unit and satisfying $\rho(b_{n},\xi_n)\geq r(b_{n})$. Proceeding as in the proof of Claim 1, we obtain the existence of $\xi_b\in N_A^P(b)$ unit such that $$\rho(b,\xi_b)\geq r(b),\;\hbox{a contradiction}.$$ Now, combining Claim 1 and Claim 2, and taking $a:=b_{\bar{N}}$, we deduce that $$a\in\bdry (\inte A),\;N_A^P(a)\not=\{0\},\;\hbox{and}\,\;\rho(a,\zeta)<r(a)\;\hbox{for all}\;\zeta\in N_A^P(a)\;\hbox{unit}.$$
Therefore, $$\begin{cases} \exists a\in\bdry(\inte A)\;\hbox{such that}\;N_A^P(a)\not=\{0\},\;\hbox{and}\;\rho(a,\zeta)<r(a),\;\forall\zeta\in N_A^P(a)\;\hbox{unit},\\ \hbox{or} \\  \exists a\in(\bdry A)\cap (\bdry(\inte A))^c\;\hbox{and}\;\exists \zeta_a\in  N_A^P(a)\;\hbox{unit such that}\; \rho(a,\zeta_a)<r(a).\end{cases}$$
\end{remark}
We proceed by introducing the four sets $\O$, $\P$, $ \bar{A}_{r}$ and $\bar{A}^{r}$, and the hypothesis (UP)$_r$. For $A\subset\R^n$ a nonempty and closed set, and $r\colon\bdry A\f (0,\infty]$ a continuous function, we consider:
\begin{enumerate}[$\bullet$]\setlength\itemsep{0.3em}
\item $\O:=\{x\in A^c : \exists a\in\proj_A(x)\cap [\bdry(\inte A)]^c\;\hbox{with}\;\|x-a\|<r(a)\}.$
\item $\P:=\{x\in A^c : \proj_A(x)\cap \bdry(\inte A)\not\subset \bdry_{\rb}(\inte A)\}$.
\item $ \bar{A}_{r}:=A\cup A_{\up\sr}\cup \O\cup\P$.
\item $\bar{A}^{r}:=A\cup A^{\up}\cup\O\cup\P$.
\item $\hbox{(UP)}{_r}\qquad\forall x\in\bar{A}_r\;\;\big[x\in\bdry \bar{A}_r\implies\proj_A(x)\;\hbox{is a singleton}\big].$
\end{enumerate}

\begin{remark}
Unlike the case $\bar{A}^r$, see Proposition \eqref{propfunbis}$(iv)$, the hypothesis \textnormal{(UP)}$_r$ is {\it not} always satisfied when $A$ is  $\bar{A}_r$-convex set. Indeed, let $A:=\{(x,y): y=0\;\hbox{or}\;y\geq4\}\subset \R^2$, and let $$r(x,y):=\begin{cases} 1\;\;\;& \hbox{if}\;
x\in\R\;\hbox{and}\;y=0,\vspace{0.1cm} \\ 3\;\;\;& \hbox{if}\;
x\in\R\;\hbox{and}\;y=4.\end{cases}$$
We have $$\bar{A}_r=\{(x,y) : y\in(-1,1)\;\hbox{or}\;y\in[2,+\infty)\}\;\,\hbox{and}\;\,\bar{A}^r=\R^2.$$ One can easily verify that $A$ is $\bar{A}_r$-convex. But, the  hypothesis \textnormal{(UP)}$_r$ is not satisfied. In fact, for any $x\in\R$, we have $$(x,2)\in\bdry \bar{A}_r=\{(x,y) : y\in\{-1,1,2\}\},\;\hbox{and}\;\,\proj_A(x,2)=\{(x,0),(x,4)\}.$$ Note that for this example, $\O=\{(x,y) : y\in(-1,0)\;\hbox{or}\;y\in(0,1)\}$, and hence, it is open. \end{remark}

The main result of this section is the following theorem in which we prove that the extended exterior $r(\cdot)$-sphere condition belongs to the $S$-convexity regularity class. This extends the main results of \cite{JCA2022}, where similar characterizations are obtained for
$\varphi$-convexity, the $\theta$-exterior sphere condition, and the $\psi$-union of closed balls property.

\begin{theorem}\label{th2} Let $A\subset\R^n$ be a nonempty and closed set, and let $r\colon\bdry A\f (0,+\infty]$ be a continuous function. Then the following assertions are equivalent$\sp\sp:$ 
\begin{enumerate}[$(i)$]
\item  $A$ satisfies the extended exterior $r(\cdot)$-sphere condition.
\item $A$ is $\bar{A}^{r}$-convex.
\item $A$ is  $\bar{A}_{r}$-convex, \textnormal{(UP)}$_r$ is satisfied, and $\O$ is open.
\end{enumerate}
\end{theorem}

Before giving the proof of Theorem \ref{th2}, we prove the following characterizations of  the sets $\bar{A}_r$ and $ \bar{A}^{r}$ introduced above. These characterizations will play an essential role in the proof of Theorem \ref{th2}.

\begin{proposition}\label{propfunbis} Let $A\subset\R^n$ be a nonempty and closed set, and let $r\colon\bdry A\f (0,\infty]$ be a continuous function. Then the following assertions hold$\sp:$
\begin{enumerate}[$(i)$]
\item $\bar{A}_{r}\subset  \bar{A}^{r}$, $\bar{A}^{r}=\bar{A}_{r}\cup A^{\up}$, $\bar{A}^{r}\cap (A_{\up\sr})^c=\bar{A}_{r}\cap (A_{\up})^c$, and $\O\subset  \bar{A}_{r}\subset \bar{A}^r$.
\item We have \begin{eqnarray*} \bar{A}_{r}=A&\cup&\bigcup_{\substack{a\in\bdry_{\rb}(\inte A)\\\zeta\in N_A^P(a)\,\textnormal{unit} }}\hspace{-0.2cm}[a,a+\rho(a,\zeta,r)\zeta)\,\cup\bigcup_{\substack{a\in(\bdry(\inte A))\cap (\bdry_{\rb}(\inte A))^c\\ \zeta\in N_A^P(a)\,\textnormal{unit}}} \hspace{-0.8cm}[a,a+\rho(a,\zeta,r)\zeta]\\ &\cup& \bigcup_{\substack{a\in(\bdry A)\cap (\bdry(\inte A))^c\\\zeta\in N_A^P(a)\,\textnormal{unit}\\ \rho(a,\zeta)\geq r(a)}}\hspace{-0.2cm}[a,a+\rho(a,\zeta,r)\zeta)\,\cup\bigcup_{\substack{a\in(\bdry  A)\cap (\bdry (\inte A))^c\\ \zeta\in N_A^P(a)\,\textnormal{unit}\\ \rho(a,\zeta)< r(a)}} \hspace{-0.8cm}[a,a+\rho(a,\zeta,r)\zeta]. \end{eqnarray*}
\item We have  \begin{eqnarray*} \bar{A}^{r}=A&\cup&\bigcup_{\substack{a\in\bdry_{\rb}(\inte A)\\\zeta\in N_A^P(a)\,\textnormal{unit} }}\hspace{-0.2cm}[a,a+\rho(a,\zeta)\zeta)\,\cup\bigcup_{\substack{a\in(\bdry(\inte A))\cap (\bdry_{\rb}(\inte A))^c\\ \zeta\in N_A^P(a)\,\textnormal{unit}}} \hspace{-0.8cm}[a,a+\rho(a,\zeta)\zeta]\\ &\cup& \bigcup_{\substack{a\in(\bdry A)\cap (\bdry(\inte A))^c\\\zeta\in N_A^P(a)\,\textnormal{unit}\\ \rho(a,\zeta)\geq r(a)}}\hspace{-0.2cm}[a,a+\rho(a,\zeta)\zeta)\,\cup\bigcup_{\substack{a\in(\bdry  A)\cap (\bdry (\inte A))^c\\ \zeta\in N_A^P(a)\,\textnormal{unit}\\ \rho(a,\zeta)< r(a)}} \hspace{-0.8cm}[a,a+\rho(a,\zeta)\zeta]. \end{eqnarray*}
\item If $A$ is $\bar{A}^r$-convex, then each point in $\bar{A}^r$ has a unique projection on $A$.
\end{enumerate}
\end{proposition}
\begin{proof} \bm{$(i)$}: Follows directly from the definitions of $\bar{A}_{r}$ and $\bar{A}^{r}$.\vspace{0.2cm}\\
\bm{$(ii)$}: We begin by proving that  $\bar{A}_{r}\subset\rhs(ii)$ (the right hand side of $(ii)$). Let $x\in \bar{A}_{r}$.\vspace{0.1cm}\\
\underline{Case 1:} $x\in A\cup A_{\up\sr}$.\vspace{0.1cm}\\
Then by Proposition \ref{propfun}$(ii)$, we have  $$x\in A\cup \bigcup_{\substack{a\in\bdry A\\\zeta\in N_A^P(a)\,\textnormal{unit}}} (a,a+\rho(a,\zeta,r),\zeta)\subset\rhs(ii). $$
\underline{Case 2:} $x\in\P$.\vspace{0.1cm}\\
Then $x\in A^c$, and there exists $a_x\in \proj_A(x)\cap \bdry(\inte A)$ such that $a_x\not\in  \bdry_{\rb}(\inte A)$. This yields that $a_x\in \bdry(\inte A)\cap(\bdry_{\rb}(\inte A))^c$ and $\zeta_x:=\frac{x-a_x}{\|x-a_x\|} \in N_A^P(a_x)$. Since $a_x\not\in \bdry_{\rb}(\inte A)$, we deduce that $$[r(a_x)=+\infty\;\hbox{and}\; \rho(a_x,\zeta_x)<+\infty]\;\hbox{or}\;[r(a_x)<+\infty\;\hbox{and}\; \rho(a_x,\zeta_x)<r(a_x)].$$
In both cases, we have $\rho(a_x,\zeta_x,r)<+\infty$ and $\rho(a_x,\zeta_x,r)<r(a_x)$. This gives that $$\rho(a_x,\zeta_x,r)=\rho(a_x,\zeta_x),\;\hbox{and hence},\;\|x-a_x\|\leq \rho(a_x,\zeta_x,r).$$ Therefore, $$x\in  \bigcup_{\substack{a\in(\bdry(\inte A))\cap (\bdry_{\rb}(\inte A))^c\\ \zeta\in N_A^P(a)\,\textnormal{unit}}} \hspace{-0.8cm}[a,a+\rho(a,\zeta,r)\zeta]\subset\rhs(ii). $$
\underline{Case 3:} $x\in\O$.\vspace{0.1cm}\\
Then $x\in A^c$, and there exists $a_x\in \proj_A(x)\cap (\bdry(\inte A))^c$ such that $\|x-a_x\|<r(a_x)$. This yields that $a_x\in(\bdry A)\cap (\bdry (\inte A))^c$ with $\|x-a_x\|<r(a_x)$ and $\zeta_x:=\frac{x-a_x}{\|x-a_x\|} \in N_A^P(a_x)$. Then, $$a_x\in(\bdry A)\cap (\bdry_{\rb} (\inte A))^c\;\,\hbox{and}\,\;\rho(a_x,\zeta_x,r)\geq \|x-a_x\|.$$
Therefore, $$x\in  \bigcup_{\substack{a\in(\bdry(\inte A))\cap (\bdry_{\rb}(\inte A))^c\\ \zeta\in N_A^P(a)\,\textnormal{unit}}} \hspace{-0.8cm}[a,a+\rho(a,\zeta,r)\zeta]\subset\rhs(ii). $$
This terminates the proof of $\bar{A}_{r}\subset\rhs(ii)$. We proceed to prove that $\rhs(ii)\subset\bar{A}_{r}$. Let $x\in \rhs(ii)$. Since $A\subset\bar{A}_{r}$,  we assume that $x\in A^c$.\vspace{0.1cm}\\
\underline{Case 1:} $\displaystyle x\in \bigcup_{\substack{a\in\bdry A\\\zeta\in N_A^P(a)\,\textnormal{unit} }}\hspace{-0.2cm}(a,a+\rho(a,\zeta,r)\zeta) $.\vspace{0.1cm}\\
Then, by Proposition \ref{propfun}$(ii)$, $x\in A_{\up\sr}\subset \bar{A}_{r}$.\vspace{0.1cm}\\
\underline{Case 2:} $x=a+\rho(a,\zeta,r)\zeta$, where $a\in(\bdry(\inte A))\cap (\bdry_{\rb}(\inte A))^c$ and $\zeta\in N_A^P(a)$ is unit.\vspace{0.1cm}\\
Then $x\in A^c$, $a\in \proj_A(x)\cap \bdry(\inte A)$, and $a\not\in\bdry_{\rb}(\inte A)$. This yields that $$\proj_A(x)\cap \bdry(\inte A)\not\subset \bdry_{\rb}(\inte A).$$
Hence, $x\in\P\subset \bar{A}_{r}$.\vspace{0.1cm}\\
\underline{Case 3:} $x=a+\rho(a,\zeta,r)\zeta$, where $a\in(\bdry A)\cap (\bdry(\inte A))^c$, $\zeta\in N_A^P(a)$ is unit, and $\rho(a,\zeta)<r(a)$. \vspace{0.1cm}\\
Then $a\in \proj_A(x)\cap (\bdry(\inte A))^c$, and $\|x-a\|=\rho(a,\zeta,r)=\rho(a,\zeta)<r(a)$. This yields that $x\in\O\subset\bar{A}_{r}$.\vspace{0.2cm}\\
\bm{$(iii)$}: The inclusion $\bar{A}^{r}\subset\rhs(iii)$ follows directly since $\bar{A}^{r}=\bar{A}_{r}\cup A^{\up}$, $ A^{\up}\subset  \rhs(iii)$ (by Proposition \ref{propfun}$(ii)$), and  $\bar{A}_{r}\subset \rhs(ii)\subset\rhs(iii)$. For the inclusion $\rhs(iii)\subset\bar{A}^{r}$, it is sufficient  to follow arguments similar to those used in the proof of $\rhs(ii)\subset\bar{A}_{r}$. \vspace{0.2cm}\\
\bm{$(iv)$}: Let $x\in \bar{A}^{r}\cap A^c$. We have, using $(iii)$, that  \begin{eqnarray*}\nonumber
 [\proj_A(x),x]&=&\bigcup_{a\in\proj_A(x)}\bigg[a,a+\|x-a\|\frac{x-a}{\|x-a\|}\bigg]\\ &=& \{x\}\cup \bigcup_{a\in\proj_A(x)}\bigg[a,a+\|x-a\|\frac{x-a}{\|x-a\|}\bigg)\label{uniqueproj}\\ &\subset & \{x\}\cup \bigcup_{a\in\proj_A(x)}\bigg[a,a+r\bigg(a,\frac{x-a}{\|x-a\|}\bigg)\frac{x-a}{\|x-a\|}\bigg)\subset \bar{A}^r.	\nonumber
 \end{eqnarray*}
Hence, by Proposition \ref{prop2}$(iii)$, we deduce that $x$ has a has a unique projection on $A$. \end{proof}

Now we are ready to provide the proof of Theorem \ref{th2}.\vspace{0.2cm}\\
{\bf Proof of  Theorem \ref{th2}.} \bm{$(i) \Longrightarrow (ii)$}: Since $A$ satisfies the extended exterior $r(\cdot)$-sphere condition, we have $\bdry_{\rb} (\inte A)=\bdry (\inte A)$, and 
$$\rho(a,\zeta)\geq r(a),\;\forall a\in(\bdry A)\cap (\bdry(\inte A))^c\,\;\hbox{and}\,\;\forall \zeta\in N_A^P(a).$$
This yields, using Proposition \ref{propfun}$(ii)$ and Proposition \ref{propfunbis}$(iii)$, that $\bar{A}^r=A\cup A^{\up}$. Therefore, by Proposition \ref{propfun}$(iii)$, we deduce that $A$ is $\bar{A}^r$-convex.\vspace{0.2cm}\\
\bm{$(ii) \Longrightarrow (iii)$}: Since $\bar{A}_r\subset \bar{A}^r$, we deduce that $A$ is $\bar{A}_r$-convex. The hypothesis (UP)$_r$ follows directly from Proposition \ref{propfunbis}$(iv)$ and the inclusion (see Proposition \ref{propfunbis}$(i)$) $$\bdry \bar{A}_{r}\subset  \bar{A}_{r}\subset  \bar{A}^{r}.$$
We proceed to prove that $\O$ is open. If not, there there exists $x_n\not\in\O$ such that $x_n\f\x\in\O$. Since $\x\in\O\subset\bar{A}^r$, we have from Proposition \ref{propfunbis}$(iv)$ that $\x$ has a unique projection $\a$ on $A$. This yields, using the definition of $\O$, that $\a\not\in \bdry(\inte A)$ and $\|\x-\a\|<r(\a)$. On the other hand, since $x_n\not\in\O$, we have that \begin{equation}\label{nearend} \|x_n-a\|\geq r(a),\;\,\forall a\in \proj_A(x_n)\cap (\bdry(\inte A))^c.\end{equation}
Let $a_n\in\proj_A(x_n)$. Since $(a_n)_n$ is bounded, it has a subsequence, we do not relabel, such that $a_n\f \tilde{a}\in\proj_A(\x)$, where the last inclusion follows since the projection map $\proj_A(\cdot)$ is closed. Hence, $\tilde{a}=\bar{a}$.\vspace{0.1cm}\\
\underline{Case 1:} $(a_n)_n$ has a subsequence, we do not relabel, with $a_n\in(\bdry A)\cap (\bdry(\inte A))^c,\;\forall n$.\vspace{0.1cm}\\
Then, using \eqref{nearend}, we have  $$\|\x-\a\|=\lim_{n\f+\infty} \|x_n-a_n\|\geq r(\a),$$
which gives the desired contradiction.\vspace{0.1cm}\\
\underline{Case 2:} $(a_n)_n$ has a subsequence, we do not relabel, with $a_n\in(\bdry A)\cap (\bdry(\inte A)),\;\forall n$.\vspace{0.1cm}\\
Then $\a\in\bdry (\inte A)$, which gives the desired contradiction.
\vspace{0.2cm}\\
\bm{$(iii) \Longrightarrow (i)$}: We proceed with a proof by contradiction. Assume that $A$ does not satisfy the extended exterior $r(\cdot)$-sphere condition. Then by Remark \ref{newrem1}, we have $$\begin{cases} \exists a\in\bdry(\inte A)\;\hbox{such that}\;N_A^P(a)\not=\{0\},\;\hbox{and}\;\rho(a,\zeta)<r(a),\;\forall\zeta\in N_A^P(a)\;\hbox{unit},\\ \hbox{or} \\  \exists a\in(\bdry A)\cap (\bdry(\inte A))^c\;\hbox{and}\;\exists \zeta_a\in  N_A^P(a)\;\hbox{unit such that}\; \rho(a,\zeta_a)<r(a).\end{cases}$$
\underline{Case 1:} $\exists a\in(\bdry A)\cap (\bdry(\inte A))^c$ and $\exists \zeta_a\in  N_A^P(a)$ unit such that $\rho(a,\zeta_a)<r(a).$\vspace{0.1cm}\\
Let $x_a=a+\rho(a,\zeta_a)\zeta_a$. For all $x\in(a,x_a]$, we have $a\in \proj_A(x)\cap (\bdry(\inte A))^c$, and $$\|x-a\|\leq \|x_a-a\|=\rho(a,\zeta_a)<r(a). $$
This yields that $$(a,x_a]\subset \O\subset \bar{A}_r.$$
Since $\O$ is open and $x_a\in\O$, there exists $\e>0$ such that $\bar{B}(x_a;\e)\subset \O$. We consider  $x_\e:=x_a+\e\zeta_a=a+(\e+\rho(a,\zeta_a))\zeta_a$. Clearly we have $$[x_a,x_\e]\subset \bar{B}(x_a;\e)\subset \O\subset \bar{A}_r.$$
Therefore, \begin{equation}\label{lastplease1} (a,x_\e]=(a_x,x_a]\cup [x_a,x_\e] \subset\O\subset \bar{A}_r.\end{equation} Now, since $x_\e\in\O$, there exists $a_\e\in\proj_A(x_\e)\cap (\bdry(\inte A))^c$ such that $\|x_\e-a_\e\|<r(a_\e)$. Hence, for all $x\in(a_\e,x_\e)$, we have $a_\e\in \proj_A(x_\e)\cap (\bdry(\inte A))^c$, and $$\|x-a_\e\|\leq \|x_\e-a_\e\|<r(a_\e).$$
Therefore,  \begin{equation}\label{lastplease2}(a_\e,x_\e]\subset\O\subset \bar{A}_r.\end{equation} 
We claim that $a\not=a_\e$. Indeed, if not, then $\zeta_a$ is realized by a $\e+\rho(a,\zeta_a)$-sphere, contradiction.  Therefore, from \eqref{lastplease1} and \eqref{lastplease2}, we have that the two normal segments to $A$, $[a,x_\e]$ and $[a_\e,x_\e]$ are inside  $\bar{A}_r$ and intersects at $x_\e$. This contradicts the $\bar{A}_r$-convexity of $A$. \vspace{0.1cm}\\
\underline{Case 2:} $\exists a\in\bdry(\inte A)$ such that $N_A^P(a)\not=\{0\},$ and $\rho(a,\zeta)<r(a)$, $\forall\zeta\in N_A^P(a)$ unit. \vspace{0.1cm}\\
Then $a\in(\bdry(\inte A))\cap (\bdry_{\rb} (\inte A))^c$. Since $N_A^P(a)\not=\{0\}$, we consider a unit vector $\zeta_a\in N_A^P(a)$. We have $\rho(a,\zeta_{a})<r(a)$. This yields that $\rho(a,\zeta_{a},r)=\rho(a,\zeta_{a})<r(a).$ Hence, for $x_{a}:=a+\rho(a,\zeta_{a},r)\zeta_{a}$, we have by Proposition \ref{propfunbis}$(ii)$ that $$[a,x_{a}]=[a,a+\rho(a,\zeta_{a},r)\zeta_{a}]\subset \bar{A}_r\,\;\hbox{and}\;\,a\in\proj_A (x_{a}).\vspace{0.1cm}$$
\underline{Case 2.1:} $\exists\e>0$ such that $[x_{a},x_{a}+\e\zeta_{a}]\subset \bar{A}_r$. \vspace{0.1cm}\\
Let $x_\e:=x_{a}+\e\zeta_{a}$, where $\e$ is taken small enough so that $x_\e\not\in A$. We have $$[a,x_\e]=[a,x_{a}]\cup[x_{a},x_\e]\subset \bar{A}_r.$$ Since $x_\e\in\bar{A}_r$ and using Proposition \ref{propfunbis}$(ii)$, there exist $a_\e\in\bdry A$ and  $\zeta_{a_\e}\in N_A^P(a_\e)$ unit, such that $$[a_
\e,x_\e)\subset [a_\e,a_\e+\rho(a_\e,\zeta_{a_\e},r)\zeta_{a_\e})\subset \bar{A}_r.$$
Since $\|x_\e-a\|=\e+\rho(a,\zeta_{a},r)$ and $\|x_\e-a_\e\|\leq \rho(a_\e,\zeta_{a_\e},r)$, we deduce that
$a_\e\not=a$. Therefore, the two normal segments to $A$, $[a,x_\e]$ and $[a_\e,x_\e]$ are inside  $\bar{A}_r$ and intersects at $x_\e$. This contradicts the $\bar{A}_r$-convexity of $A$.\vspace{0.1cm}\\
\underline{Case 2.2:} $\exists x_n:=x_{a}+t_n\zeta_{a}$, $n\in\N$, such that $t_n>0$ and $x_n\not\in \bar{A}_r$ for all $n$, and $x_n\f x_{a}$.\vspace{0.1cm}\\ 
Since $x_n\not\in  \bar{A}_r$, $x_{a}\in  \bar{A}_r$, and $x_n\f x_{a}$, we deduce that $x_{a}\in\bdry(\bar{A}_r)$. Hence, by (UP)$_r$, $x_{a}$ has a unique projection $a$ on $A$. Let $a_n\in\proj_A(x_n)$. Since $x_n\f x_{a}$, we have that $(a_n)_n$ is bounded.  Hence, it admits a subsequent, we do not relabel, that converges to a point $\tilde{a}$. Having that the projection map $\proj_A(\cdot)$ is closed and that $\proj_A(x_{a})=\{a\}$, we deduce that $\tilde{a}=a$, and then $a_n\f a$. Add to this that  $\bdry_{r}(\inte A)$ is closed and $a\not\in \bdry_{r}(\inte A)$, we conclude that, for $n$ sufficiently large, $a_n\in (\bdry_{r}(\inte A))^c$. \vspace{0.1cm}\\ 
\underline{Case 2.2.1:} $(a_n)_n$ has a subsequence $(a_{n_k})_k$ such that $a_{n_k}\in\bdry(\inte A)$ for all $k$.\vspace{0.1cm}\\
Then, for $k$ sufficiently large, we have  $\proj_A(x_{n_k})\cap (\bdry(\inte A))\not\subset\bdry_{r} (\inte A)$. Hence, using that $x_{n_k}\in A^c$ (since $x_{n_k}\not\in \bar{A}_r$ and $A\subset \bar{A}_r$), we obtain that, for $k$ sufficiently large, $x_{n_k}\in\P\subset\bar{A}_r$, a contradiction.\vspace{0.1cm}\\
\underline{Case 2.2.2:} $\exists N$ such that for $n\geq N$, we have $a_n\in (\bdry(\inte A))^c$.\vspace{0.1cm}\\
We define $\zeta_n:=\frac{x_n-a_n}{\|x_n-a_n\|}$. So, we have  $\zeta_n\f \frac{x_{a}-a}{\|x_{a}-a\|}=\zeta_{a}.$\vspace{0.1cm}\\
\underline{Case 2.2.2.1:} $\rho(a_n,\zeta_n)\geq r(a_n)$ for all  $n\geq N$.\vspace{0.1cm}\\
Then, taking $n\f+\infty$ in $\rho(a_n,\zeta_n)\geq r(a_n)$, and using the proximal normal inequality, we deduce that $\rho(a,\zeta_{a})\geq r(a)$, contradiction.\vspace{0.1cm}\\
\underline{Case 2.2.2.2:} $\exists n_0\geq N$ such that $\rho(a_{n_0},\zeta_{n_0})<r(a_{n_0})$.\vspace{0.1cm}\\
Then, we have $a_{n_0}\in (\bdry(\inte A))^c$ with $\rho(a_{n_0},\zeta_{n_0})<r(a_{n_0})$. Hence, as in Case 1, where we replace $a$ by $a_{n_0}$, we obtain the desired contradiction. 

The proof of  Theorem \ref{th2} is terminated. \hfill $\blacksquare$ 

\begin{remark} Theorem \ref{th2} coincides with \cite[Theorem 3.15]{JCA2022} when $A$ is regular closed, that is, $A=\clo(\inte A)$. Indeed, when $A$ is regular closed, the extended exterior $r(\cdot)$-sphere condition coincides with the $\theta$-exterior sphere condition, for $\theta(\cdot):=\frac{1}{2r(\cdot)}$. Moreover, we have $$\bar{A}_{r}=\bar{A}_{\theta},\;\bar{A}^{r}=\bar{A}^{\theta},\;\hbox{and}\;\,\O=\emptyset\;\hbox{is open}.$$
\end{remark}

\end{document}